\documentclass{amsart}

\usepackage{hyperref}
\usepackage{array}
\usepackage{amssymb, latexsym, url}
\usepackage{amsthm, epsfig}
\usepackage{xcolor}
\usepackage{mathtools}

\DeclareMathOperator{\Aut}{\mathrm Aut}
\DeclareMathOperator{\CM}{\mathrm CM}

\DeclareMathOperator{\Hom}{\mathrm Hom}

\DeclareMathOperator{\Tor}{\mathrm Tor}

\begin{document}
 \bibliographystyle{plain}

 \newtheorem{theorem}{Theorem}[section]
 \newtheorem{lemma}{Lemma}[section]
 \newtheorem{corollary}{Corollary}[section]
 \newtheorem{conjecture}{Conjecture}[section]
 \newtheorem{question}{Question}[section]
 
 \newcommand\house[1]{%
     \begingroup\setlength\arraycolsep{0pt}
     \begin{array}[t]{@{}c@{} | c |@{}c@{}}
     \firsthline
     &\;#1\;{}&
     \end{array}
     \endgroup
 }
 \newcommand{\mc}{\mathcal}
 \newcommand{\mbb}{\mathbb}
 \newcommand{\fA}{\mathfrak A}
 \newcommand{\B}{\mc B}
 \newcommand{\fB}{\mathfrak B}
 \newcommand{\cC}{\mc C}
 \newcommand{\D}{\mc D}
 \newcommand{\E}{\mc E}
 \newcommand{\F}{\mc F}
 \newcommand{\G}{\mc G}
 \newcommand{\hH}{\mc H}
 \newcommand{\fG}{\mathfrak G}
 \newcommand{\fI}{\mathfrak I}
 \newcommand{\I}{\mc I}
 \newcommand{\J}{\mc J}
 \newcommand{\K}{\mc K}
 \newcommand{\lL}{\mc L}
 \newcommand{\M}{\mc M}
 \newcommand{\fM}{\mathfrak M}
 \newcommand{\pp}{\mc P}
 \newcommand{\fP}{\mathfrak P}
 \newcommand{\fR}{\mathfrak R}
 \newcommand{\rR}{\mc R}
 \newcommand{\fS}{\mathfrak S}
 \newcommand{\sS}{\mc S}
 \newcommand{\U}{\mc U}
 \newcommand{\uU}{\mathfrak U}
 \newcommand{\X}{\mc X}
 \newcommand{\Y}{\mc Y}
 \newcommand{\A}{\mathbb{A}}
 \newcommand{\C}{\mathbb{C}}
 \newcommand{\pP}{\mathbb{P}}
 \newcommand{\Q}{\mathbb Q}
 \newcommand{\R}{\mathbb R}
 \newcommand{\T}{\mathbb T}
 \newcommand{\Z}{\mathbb{Z}}
 \newcommand{\ahat}{\widehat\alpha}
 \newcommand{\bhat}{\widehat\beta}
 \newcommand{\fhat}{\widehat f}
 \newcommand{\ghat}{\widehat g}
 \newcommand{\hhat}{\widehat h}
 \newcommand{\wsigma}{\widetilde{\sigma}}
 \newcommand{\wtau}{\widetilde{\tau}}
 \newcommand{\p}{\boldsymbol{\varphi}}
 \newcommand{\h}{\tfrac{1}{2}}
 \newcommand{\hh}{\frac{1}{2}}
 \newcommand{\ba}{\boldsymbol a}
 \newcommand{\bb}{\boldsymbol b}
 \newcommand{\be}{\boldsymbol e}
 \newcommand{\bm}{\boldsymbol m}
 \newcommand{\bn}{\boldsymbol n}
 \newcommand{\bu}{\boldsymbol u}
 \newcommand{\bv}{\boldsymbol v}
 \newcommand{\bw}{\boldsymbol w}
 \newcommand{\bx}{\boldsymbol x}
 \newcommand{\bwy}{\boldsymbol y}
 \newcommand{\bL}{\boldsymbol L}
 \newcommand{\bta}{\boldsymbol \beta}
 \newcommand{\bet}{\boldsymbol \eta}
 \newcommand{\bxi}{\boldsymbol \xi}
 \newcommand{\bo}{\boldsymbol 0}
 \newcommand{\bid}{\boldsymbol 1}
 \newcommand{\ep}{\varepsilon}
 \newcommand{\vphi}{\varphi}
  \newcommand{\eeta}{\mu}
 \newcommand{\dlambda}{\text{\rm d}\lambda}
 \newcommand{\dbeta}{\text{\rm d}\beta}
 \newcommand{\dmu}{\text{\rm d}\mu}
 \newcommand{\dr}{\text{\rm d}r}
 \newcommand{\du}{\text{\rm d}u}
 \newcommand{\dv}{\text{\rm d}v}
 \newcommand{\dt}{\text{\rm d}t}
 \newcommand{\dw}{\text{\rm d}w}
 \newcommand{\dx}{\text{\rm d}x}
 \newcommand{\dy}{\text{\rm d}y}
 \newcommand{\dxi}{\text{\rm d}\xi}
 \newcommand{\oQ}{\overline{\Q}}
 \newcommand{\oq}{\oQ^{\times}}
 \newcommand{\oQt}{\oQ^{\times}/\Tor\bigl(\oQ^{\times}\bigr)}
 \newcommand{\ot}{\Tor\bigl(\oQ^{\times}\bigr)}

\def\housealp{{%
    \setbox2=\hbox{$\alpha$}
    \vrule height \dimexpr\ht2+1.75pt width .4pt depth \dp2\relax
    \vrule height 6.05pt width \dimexpr\wd2+2pt depth -5.65pt
    \llap{$\alpha$\kern1pt}
    \vrule height \dimexpr\ht2+1.75pt width .4pt depth \dp2\relax
}}
\def\houseeta{{%
    \setbox2=\hbox{$\eeta$}
    \vrule height \dimexpr\ht2+1.75pt width .4pt depth \dp2\relax
    \vrule height 6.05pt width \dimexpr\wd2+2pt depth -5.65pt
    \llap{$\eeta$\kern1pt}
    \vrule height \dimexpr\ht2+1.75pt width .4pt depth \dp2\relax
}}

\title[heights of generators]
{Small integral  generators of totally complex number fields}
\author{Shabnam Akhtari, Jeffrey~D.~Vaaler and Martin Widmer}
\address{Department of Mathematics, Pennsylvania State University, University Park, PA 16802 USA}
\email{akhtari@psu.edu}

\address{Department of Mathematics, University of Texas, Austin, TX 78712 USA}
\email{vaaler@math.utexas.edu}

\address{Graz University of Technology, Institute of Analysis and Number Theory, Steyrergasse 30/II, 8010 Graz, Austria}
\email{martin.widmer@tugraz.at}

\allowdisplaybreaks
\numberwithin{equation}{section}

\begin{abstract}  Let $K$ be an algebraic number field and $H$ the absolute Weil height.  Write $c_K$ for a certain positive constant that is 
an invariant of $K$.  We consider the question: does $K$ contain an algebraic integer $\alpha$ such that both $K = \Q(\alpha)$ and 
$H(\alpha) \le c_K$?  If $K$ has a real embedding then a positive answer was established in previous work.  Here we obtain a positive 
answer if $\Tor\bigl(K^{\times}\bigr) \not= \{\pm 1\}$, and so $K$ has only complex embeddings.  We also show that if the answer is negative, 
then $K$ is totally complex, $\Tor\bigl(K^{\times}\bigr) = \{\pm 1\}$, and $K$ is a Galois extension of its maximal totally real subfield.
Further, we show that if $\mu \in O_K$ is not totally real, then there exists $\alpha$ in $O_K$ with $K = \Q(\alpha)$ and 
$H(\alpha) \le H(\mu)\thinspace c_K$. 
\end{abstract}

\maketitle


\section{Introduction}

Let $K$ be an algebraic number field of degree $d = [K : \Q]$, and let $\Delta_K$ be the discriminant of $K$.  We define the positive constant
\begin{equation}\label{early3}
c_K = \biggl(\frac{2}{\pi}\biggr)^{s/d} \bigl|\Delta_K\bigr|^{1/2d},
\end{equation}  
where $s$ is the number of complex places of $K$.  We write $K^{\times}$ for the multiplicative group of nonzero elements of $K$, and
\begin{equation}\label{early7}
H : K^{\times} \rightarrow [1, \infty)
\end{equation}  
for the absolute, multiplicative Weil height defined in (\ref{extra261}).

In \cite[Question 2]{ruppert1998} W.~Ruppert asked the following question:

\begin{question} {\sc [Ruppert, 1998]}\label{con1}  Does there exist a positive constant $A = A(d)$ such that if $K$ is an algebraic 
number field of degree $d$ over $\Q$, then there exists an element $\alpha$ in $K$ such that 
\begin{equation*}\label{early12}
K = \Q(\alpha),\quad\text{and}\quad H(\alpha) \le A \bigl|\Delta_K\bigr|^{1/2d}\thinspace ?
\end{equation*} 
\end{question}

Ruppert stated his question using the naive height.  However, it follows from elementary inequalities for heights that the variant we 
have stated here is equivalent to the question originally asked by Ruppert.  It is known that there always exists a generator whose height is at most $|\Delta_K|^{1/d}$, see, e.g., \cite[Lemma 7.1]{PazukiWidmer2021}.

In \cite[Proposition 2]{ruppert1998} Ruppert obtained a positive 
answer to his question when $[K : \Q] = 2$.  He also proved that if $K$ is a real quadratic extension of $\Q$, then the generator 
$\alpha$ can be selected from the ring $O_K$ of algebraic integers in $K$.  In \cite{vaaler2013} the second and third named authors provided the 
following partial answer to Ruppert's question:

\begin{theorem}\label{thmearly1}  Assume that $K$ has an embedding into $\R$.  Then there exists an algebraic integer $\alpha$ in $O_K$
such that
\begin{equation}\label{early16}
K = \Q(\alpha),\quad\text{and}\quad H(\alpha) \le c_K.
\end{equation}
\end{theorem}

In Theorem \ref{thmearly1} the generator $\alpha$ is an algebraic integer, a requirement that was {\it not} stated in Ruppert's question, 
while the height of $\alpha$ is bounded in a manner that was anticipated in Ruppert's question. Hence Theorem \ref{thmearly1} generalizes 
Ruppert's earlier result to number fields $K$ that have at least one real embedding.
 
Ruppert's question has also been answered in the  affirmative for abelian number fields of sufficiently large degree \cite[Theorem 1]{Widmer2025}. Moreover, it is known that the exponent $1/2d$ cannot be replaced by a smaller value,
at least not if $d$ is even \cite[Theorem 2]{Widmer2025}.

In this note we establish a positive answer to Ruppert's question for a new class of number fields. 

\begin{theorem}\label{thmearly2}  Assume that $K$ is a number field, and $\mu$ is an algebraic integer in $O_K$ that is not totally
real.  Then there exists an algebraic integer $\alpha$ in $O_K$ such that 
\begin{equation}\label{early22}
K = \Q(\alpha),\quad\text{and}\quad H(\alpha) \le H(\mu)\thinspace c_K.
\end{equation}
\end{theorem}

Let $\Tor\bigl(K^{\times}\bigr)$ denote the torsion subgroup of the multiplicative group $K^{\times}$.  This group is known to be a finite, cyclic group 
of even order $2 q_K$, where $q_K$ is a positive divisor of $\Delta_K$ (see \cite[Proposition 3.11]{Narkieweicz2004}).
The following simple corollary illustrates how Theorem \ref{thmearly2}
can be applied.

\begin{corollary}\label{corearly1}  Assume that $K$ is a number field such that 
\begin{equation*}
\Tor\bigl(K^{\times}\bigr) \neq  \{\pm 1\}.
\end{equation*}
Then there exists $\alpha$ in $O_K$ such that 
\begin{equation*}
K = \Q(\alpha), \quad \text{and}\quad H(\alpha) \le c_K.
\end{equation*}
\end{corollary}

To derive Corollary \ref{corearly1} from Theorem \ref{thmearly2} we select $\mu$ to be an element of the group $\Tor\bigl(K^{\times}\bigr)$
having order greater than or equal to $3$.   It follows that $\mu$ is a complex (and not real) root of unity and therefore $H(\mu) = 1$.  In this 
example, $\mu$ has no real conjugates over $\Q$ and therefore the field $K$ containing $\mu$ is totally complex.

Another implication of  Theorem \ref{thmearly2} is an affirmative answer to Question \ref{con1}
(with $A = 2$, say) 
whenever the field 
$K$ contains a non-quadratic algebraic integer of the form $n^{1/m}$ with integers $m \geq n \geq 2$. In Theorem \ref{thmearly2}  we may take $\mu = n^{1/m}$, for which we have $H(\mu) \leq 2$.


The constructive method used in the proof of Theorem \ref{thmearly2} enables us to obtain a result in which we identify a class of 
number fields that might not have a small integral generator. 

\begin{theorem}\label{thmearly3}   Let $\hH \ge 1$ and assume that $K$ is a number field such that
\begin{equation}\label{early35}
\hH c_K < \min\big\{H(\alpha) : \text{$\alpha \in O_K$ and $K = \Q(\alpha)$}\big\}.
\end{equation}
Let $F \subseteq K$ be the maximal, totally real subfield of $K$.  Then $K$ is totally complex, $K/F$ is Galois, and every $\mu \in O_K$ 
with $H(\mu) \leq \hH$ is contained in $F$.
\end{theorem}

For a number field  $K$ that satisfies the hypotheses of Theorem \ref{thmearly3}, a representation for the Galois group $\Aut(K/F)$ is 
provided in (\ref{extra755}).  We also note that the restriction of complex conjugation to $K$ belongs to $\Aut(K/F)$ and is an automorphism
of order $2$.  This implies that  $k = K \cap \R$ is the unique real subfield of $K$ with $[K:k]=2$  and $F \subseteq k \subseteq K$. 
A more general form of these remarks follows from 
(\ref{extra719}) and (\ref{extra726}) in the proof of Theorem \ref{thmearly3}.

Taking $\hH=1$ in Theorem \ref{thmearly3}, we conclude that if $K$ is a number field such that
\begin{equation}\label{early48}
c_K < \min\big\{H(\alpha) : \text{$\alpha \in O_K$ and $K = \Q(\alpha)$}\big\},
\end{equation}
then $K$ is totally complex, $K$ is Galois over its maximal totally real subfield $F$,  and $\Tor\bigl(K^{\times}\bigr) =  \{\pm 1\}$.

In Section \ref{egs} we record some examples 
which motivate our new results. In Section \ref{lems} we prove some auxiliary lemmas. 
The proof of  our main results Theorem \ref{thmearly2} and Theorem \ref{thmearly3} are completed in Sections \ref{Tmain1} and \ref{Tmain2}, respectively.

\section{Examples}\label{egs}

\subsection{Imaginary quadratic fields.} Let $m$ be a squarefree, negative integer and let $L = \Q\bigl(\sqrt{m}\bigr)$ be the imaginary quadratic field generated by $\sqrt{m}$.  
An integral basis for the ring $O_L$ is well known (see \cite[Theorem 7.1.1]{alaca2004}).  We also recall
(see \cite[Theorem 7.1.2]{alaca2004}) that
\begin{equation}\label{early62}
\Delta_L = \begin{cases}   m&    \text{if $m \equiv 1 \pmod 4$},\\
				       4m&    \text{if $m \not\equiv 1 \pmod 4$}.\end{cases}	
\end{equation}
It is now a simple matter to minimize the height over elements of $O_L$ that generate the field $L$ by constructing their minimal polynomials (see \cite[Section 2]{ruppert1998} for more details).   If $m \equiv 1 \pmod 4$ we find that
\begin{equation}\label{early69}
\min\big\{H(\alpha) : \text{$\alpha \in O_L$ and $L = \Q(\alpha)$}\big\} = \h \bigl(1 + |\Delta_L|\bigr)^{\hh},
\end{equation}
and if $m \not\equiv 1 \pmod 4$ then 
\begin{equation}\label{early76}
\min\big\{H(\alpha) : \text{$\alpha \in O_L$ and $L = \Q(\alpha)$}\big\} = \h |\Delta_L|^{\hh}.
\end{equation}

Among the imaginary quadratic fields $L$, only the fields 
\begin{equation*}
L = \Q\bigl(\sqrt{-3})\quad \text{and}\quad L = \Q\bigl(\sqrt{-1}\bigr)
\end{equation*}
satisfy the condition $\Tor\bigl(L^{\times}\bigr) \not= \{\pm 1\}$.  These fields are both generated by a root of unity and therefore the minimal height
of an algebraic integer that generates the field is $1$.  This conclusion also follows from (\ref{early62}), (\ref{early69}), and (\ref{early76}).  

If $L = \Q\bigl(\sqrt{m}\bigr)$ satisfies $\Tor\bigl(L^{\times}\bigr) = \{\pm 1\}$ then it follows from our previous remarks that the minimum height of
an integral generator is greater than $1$.  As the value of the minimal height is given by (\ref{early69}) or by (\ref{early76}), it is easy to verify that
\begin{equation*}
c_L = \biggl(\frac{2}{\pi}\biggr)^{1/2} \bigl|\Delta_L\bigr|^{1/4}< \min\big\{H(\alpha) : \text{$\alpha \in O_L$ and $L = \Q(\alpha)$}\big\}
\end{equation*}  
in both cases.  This shows that for imaginary quadratic fields $L$, the inequality \eqref{early48}
holds if and only if $L$ satisfies
$\Tor\bigl(L^{\times}\bigr) = \{\pm 1\}$. 

\subsection{A quartic number field.} While for quadratic number fields $K$, the inequality \eqref{early48} holds if and only if $K$ is a totally complex  Galois extension of $\Q$  and $\Tor\bigl(K^{\times}\bigr) =  \{\pm 1\}$,  for number fields of degree greater than $2$,  this equivalence no longer holds.
The following is an example of a totally complex quartic number field $K$ that is a Galois  extension of a totally real number field,  with 
$\Tor\bigl(K^{\times}\bigr) = \{\pm 1\}$ for which (\ref{early48}) does not hold.  Let $\alpha$ be an algebraic integer that satisfies
\begin{equation*}
\alpha^2= \sqrt{3} - 2 < 0,
\end{equation*}
and $K=\Q(\alpha)$.  The quartic field $K$ has discriminant (see \cite[Theorem 1]{huardspearmanwilliams1995})
\begin{equation*}\label{early104}
\Delta_K=2^8 3^2.
\end{equation*}
On the other hand,
\begin{equation*}
H(\alpha)=H(\sqrt{3}-2)^{1/2}=(\sqrt{3}+2)^{1/4} < (2/\pi)^{1/2}(2^83^2)^{1/8} = c_K.
\end{equation*}
Assume that $\Tor\bigl(K^{\times}\bigr) \neq \{\pm 1\}$. Since $2$ and $3$ are the only primes that ramify in $K$ we conclude that 
\begin{equation*}
\h \bigl(\sqrt{-3}-1\bigr) \in K\quad \text{or}\quad \sqrt{-1} \in K.
\end{equation*}
But since $\sqrt{3}\in K$,  we get  $\sqrt{-1} \in K$ in either case, and thus also
$\sqrt{2-\sqrt{3}}\in K$. However, the algebraic number $\sqrt{2-\sqrt{3}}$ is totally real of degree $4$, and therefore it cannot belong to the totally 
complex quartic field $K$.  This contradiction confirms that $\Tor\bigl(K^{\times}\bigr) = \{\pm 1\}$.
 
\subsection{ A family of $\CM$-fields.} We recall that $K$ is a $\CM$-field if $K$ has only complex embeddings and there exists a totally real subfield $k \subseteq K$ such that $K/k$ is 
a quadratic extension.  There are well-known characterizations of $\CM$-fields due to Blanksby and Loxton \cite{blanksby1978} and
Shimura \cite[Proposition 5.11 ]{shimura1971}.  It follows that if $K$ is a $\CM$-field that satisfies the hypotheses of Theorem \ref{thmearly3} 
then the maximal totally real subfield $F$ that occurs in the statement of Theorem \ref{thmearly3} is equal to $k$.  

Next we give a family of examples of $\CM$-fields of degree $d = 2N \geq 4$ that satisfy \eqref{early48}.  First we pick  a totally 
real field $F$ of degree $N$. Let $\{\omega_1, \ldots, \omega_N\}$ be a $\Z$-basis of $O_F$.  Let $n$ be a positive, squarefree integer, 
set $M=\Q(\sqrt{-n})$, and let $1, \xi$ be a  $\Z$-basis of $O_M$.  Assume that $\Delta_F$ and $\Delta_M$ are coprime, and set 
$K=FM=F(\sqrt{-n})$.  This implies that 
\begin{equation*}
\{\omega_1, \dots, \omega_N, \omega_1\xi, \dots, \omega_N\xi\}
\end{equation*}
is a  $\Z$-basis of $O_K$, and
\begin{equation*}
|\Delta_K|= \Delta_F^2 |\Delta_M|^N\leq \Delta_F^2(4n)^N.
\end{equation*}
In particular, every $\alpha\in O_K$ with $K=\Q(\alpha)$ can be written in the form 
\begin{equation*}
\alpha=\tfrac{1}{2} \bigl(\omega+\omega'\sqrt{-n}\bigr),
\end{equation*}
where $\omega$ and $\omega' \not= 0$ are in $O_F$. 
Using that $F$ is totally real, and writing $\tau$ for an embedding in $\Hom(K, \C)$, we find
\begin{equation*}
\begin{split}
H(\alpha) &\geq \prod_{\tau \in \Hom(K, \C)} |\tau(\alpha)|^{1/d} = \h \prod_{\tau \in \Hom(K, \C)} \bigl|\tau(\omega)+\tau(\omega')\tau(\sqrt{-n})\bigr|^{1/d}\\
                &\geq \h \sqrt{n}\thinspace \prod_{\tau \in \Hom(K, \C)}\bigl|\tau(\omega')\bigr|^{1/d}\geq \h \sqrt{n}.
\end{split}
\end{equation*}
Hence $K$ satisfies \eqref{early48}
for all sufficiently large $n$.

\section{Lemmas}\label{lems}

Let $K$ be an algebraic number field. We use
\begin{equation*}\label{extra249}
|\ | : \C \rightarrow [0, \infty)\quad\text{and}\quad \rho : \C \rightarrow \C,
\end{equation*}
for the 
standard  absolute value on $\C$ and complex conjugation, respectively.
At each place $v$ of $K$ we write $K_v$ for the completion of $K$ with respect to an absolute value from $v$.  Then $\|\ \|_v$ is the unique absolute 
value from $v$ that extends either the usual Euclidean absolute value on $\Q_{\infty}$ or the unique $p$-adic absolute value on $\Q_p$.  We write
$d = [K : \Q]$ for the degree of $K$ and $d_v = [K_v : \Q_v]$ for the local degree at the place $v$.  Then we define a second absolute value from $v$ by
\begin{equation}\label{extra258}
|\ |_v = \|\ \|_v^{d_v/d}.
\end{equation}
It follows that the absolute, multiplicative Weil height of $\alpha$ in $K^{\times}$ is the map (\ref{early7}) defined by 
(see \cite[section 1.5.7]{bombieri2006})
\begin{equation}\label{extra261}
H(\alpha) = \prod_v \max\big\{1, |\alpha|_v\big\}.
\end{equation}

We also write $W_{\infty}(K/\Q)$ for the set of archimedean places of $K$.  If $v$ is a real place of $K$ then $\sigma_v : K \rightarrow \C$ 
is the embedding associated to $v$ by
\begin{equation}\label{extra265}
\|\beta\|_v = \bigl|\sigma_v(\beta)\bigr|\quad\text{for each $\beta$ in $K$}.
\end{equation}
If $w$ is a complex place of $K$ then
\begin{equation*}
\sigma_w : K \rightarrow \C\quad\text{and}\quad \rho \sigma_w : K \rightarrow \C
\end{equation*}
are the two distinct embeddings associated to $w$ by the identity
\begin{equation}\label{extra274}
\|\beta\|_w = \bigl|\sigma_w(\beta)\bigr| = \bigl|\rho \sigma_w(\beta)\bigr| \quad\text{for each $\beta$ in $K$}.
\end{equation}
We write more simply $\Hom(K, \C)$ for the collection of all $d$ distinct embeddings of $K$ into $\C$.
If $\alpha \not= 0$ belongs to the ring $O_K$ of algebraic integers in $K$ then from (\ref{extra258}), (\ref{extra265}), and (\ref{extra274}), 
we obtain the identity
\begin{equation*}
H(\alpha)^d = \prod_{v | \infty} \max\big\{1, \|\alpha\|_v^{d_v}\big\} = \prod_{\tau \in \Hom(K, \C)} \max\big\{1, |\tau \alpha|\big\}.
\end{equation*}

The following result is a version of Minkowski's basic theorem on lattice points in convex bodies.  A proof is given in \cite[Chapter I, Theorem 5.3]{neukirch1999}.

\begin{theorem}\label{thmhaar3}  For each embedding $\tau$ in $\Hom(K, \C)$ let
$b(\tau)$ be a positive real number.  Assume that $b(\rho \tau) = b(\tau)$ for each $\tau$ in $\Hom(K, \C)$, and
\begin{equation}\label{extra310}
(c_K)^d < \prod_{\tau \in \Hom(K, \C)} b(\tau),
\end{equation} 
where $c_K$ is defined in {\rm (\ref{early3})}.  Then there exists $\xi \not= 0$ in $O_K$ such that
\begin{equation*}
|\tau \xi| < b(\tau)\quad\text{for all $\tau$ in $\Hom(K, \C)$}.
\end{equation*}
\end{theorem}

The next lemma is a simple application of Minkowski's theorem.

\begin{lemma}\label{lemhaar1}  Assume that the field $K$ has at least two archimedean places.  Let $w$ be an archimedean place of $K$.  For each archimedean place $v\neq w$ let 
$0<B_v\leq 1$.
Then there exists a nonzero algebraic integer $\xi^{(w)}$ in $O_K$ such that
\begin{equation}\label{extra338}
\bigl|\xi^{(w)}\bigr|_v < B_v\quad\text{if $v | \infty$ and $v \not= w$,}
\end{equation}
and
\begin{equation}\label{extra345}
 \bigl|\xi^{(w)}\bigr|_w  \le  c_K \prod_{\substack{v\mid \infty \\ v\neq w}}B_v^{-1}.
\end{equation}
Furthermore, we have $H\bigl(\xi^{(w)}\bigr)=\bigl|\xi^{(w)}\bigr|_w$.
\end{lemma}

\begin{proof}  Let $\ep > 0$.  If $v \not= w$ is a real place of $K$ and 
\begin{equation*}
\sigma_v : K \rightarrow \C 
\end{equation*}
is the associated embedding, we set
\begin{equation*}
b\bigl(\sigma_v\bigr) = B_v^d.
\end{equation*}  
Similarly, if $v \not= w$ is a complex place of $K$ and the associated embeddings are
\begin{equation*}
\sigma_v : K \rightarrow \C\quad\text{and}\quad \rho \sigma_v : K \rightarrow \C,
\end{equation*}
we set
\begin{equation*}
b\bigl(\sigma_v\bigr) = b\bigl(\rho \sigma_v\bigr) = B_v^{d/2}.
\end{equation*}
If $w$ is a real place we define
\begin{equation*}
b\bigl(\sigma_w\bigr) = \big((1 + \ep)c_K \prod_{\substack{v\mid \infty \\ v\neq w}} B_v^{-1}\big)^d,
\end{equation*}
and if $w$ is a complex place, we define
\begin{equation*}
b\bigl(\sigma_w\bigr) = b\bigl(\rho \sigma_w\bigr) = \big((1 + \ep) c_K \prod_{\substack{v\mid \infty \\ v\neq w}} B_v^{-1} \big)^{d/2}.
\end{equation*}
Then it follows in all cases that
\begin{equation}\label{extra390}
\prod_{\tau \in \Hom(K, \C)} b(\tau) = \left((1 + \ep)c_K\right)^d.
\end{equation}

The identity (\ref{extra390}) implies that the positive real numbers $b(\tau)$ satisfy the hypothesis (\ref{extra310}) in the statement of 
Theorem \ref{thmhaar3}.  Then it follows from the conclusion of Theorem \ref{thmhaar3} that there exists an algebraic integer $\xi^{(w)} \not= 0$ 
in $O_K$ such that 
\begin{equation*}
\bigl|\tau \xi^{(w)}\bigr| <  b(\tau)\quad\text{for all $\tau$ in $\Hom(K, \C)$}.
\end{equation*}
We have shown that for every positive value of $\ep$ there exists a nonzero algebraic integer $\xi^{(w)}$ in $O_K$ such that
\begin{equation}\label{extra452}
\bigl|\xi^{(w)}\bigr|_v < B_v\quad\text{if $v | \infty$ and $v \not= w$,}
\end{equation}
and
\begin{equation}\label{extra459}
\bigl|\xi^{(w)}\bigr|_w  < (1 + \ep)c_K \prod_{\substack{v\mid \infty \\ v\neq w}} B_v^{-1}.
\end{equation}
Using that $B_v\leq 1$ whenever $v\mid \infty$ and $v\neq w$, and that $\xi^{(w)}$ is a nonzero algebraic integer, we conclude from the product formula that $\bigl|\xi^{(w)}\bigr|_w>1$. Thus 
(\ref{extra261}) gives
\begin{equation*}
H\bigl(\xi^{(w)}\bigr)=\bigl|\xi^{(w)}\bigr|_w.
\end{equation*}
It follows from Northcott's theorem in \cite{northcott1949} that the set of nonzero algebraic integers $\xi^{(w)}$ in $O_K$ that satisfy (\ref{extra452}) and
(\ref{extra459}) with $\ep = 1$ is finite.  And we have shown that this finite set is not empty for every positive value of $\ep$.  Therefore the 
statement of the lemma follows.
\end{proof}

\begin{lemma}\label{lemhaar2}  Assume that the field $K$ is totally complex and has at least two archimedean places.
Let $w$ be an archimedean place of $K$ and let $\xi^{(w)}$ be a nonzero element of $O_K$ that satisfies 
\begin{equation}\label{extra498}
|\xi^{(w)}|_v < 1 \quad\text{if $v | \infty$ and $v \not= w$.}
\end{equation}
Write
\begin{equation*}
\sigma_w : K \rightarrow \C\quad \text{and}\quad \rho \sigma_w : K \rightarrow \C
\end{equation*}
for the embeddings of $K$ into $\C$ that satisfy the identity {\rm (\ref{extra274})}.  Assume that the field
\begin{equation*}
k = \Q(\xi^{(w)})
\end{equation*}
is a proper subfield of $K$.   Then we have 
\begin{equation}\label{extra510}
\bigl[K : k] = 2.  
\end{equation}
Moreover, the subfield $k$ satisfies 
\begin{equation}\label{extra516}
\sigma_w(k) = \sigma_w(K) \cap \R.
\end{equation}
And the restriction of complex conjugation
\begin{equation}\label{extra523}
\rho : \sigma_w(K) \rightarrow \sigma_w(K)
\end{equation}
is an automorphism that fixes the subfield $\sigma_w(k)$.
\end{lemma}

\begin{proof}  
Since $\xi^{(w)}$ is a nonzero algebraic integer it follows from (\ref{extra498}) and the product formula
that 
\begin{equation*}
1 < |\xi^{(w)}|_w.
\end{equation*}
Let $x$ be the unique archimedean place of the proper subfield $k$ that 
satisfies $w | x$.  Write $K_w$ for the completion of $K$ at the place $w$ and $k_x$ for the completion of $k$ at the place $x$.  
Then we get
\begin{equation*}
1 < |\xi^{(w)}|_w = |\xi^{(w)}|_x.
\end{equation*}
If $v \not= w$ is a second archimedean place of $K$ that also satisfies $v | x$, we find that
\begin{equation*}
1 < |\xi^{(w)}|_x=|\xi^{(w)}|_v
\end{equation*}
which contradicts (\ref{extra498}).  We conclude that $w$ is the {\it unique} place of $K$ that satisfies $w | x$.  Because 
the global degree of the extension $K/k$ is the sum of local degrees over the completion $k_x$, we conclude that
\begin{equation}\label{extra558}
2 \le [K : k] = \bigl[K_w : k_x\bigr] \le [\C : \R] = 2.
\end{equation}
This verifies (\ref{extra510}).  

It also follows from the equality in (\ref{extra558}) that the completion $k_x$  is isomorphic 
to $\R$.  We conclude that
\begin{equation}\label{extra565}
\sigma_w\bigl(k\bigr) \subseteq \R,\quad\text{and}\quad \sigma_w\bigl(k\bigr) \subseteq \sigma_w(K) \cap \R \subseteq \sigma_w(K).
\end{equation}
Because $K/k$ is a quadratic extension it follows from the two containments on the right of (\ref{extra565}) that either
\begin{equation}\label{extra572}
\sigma_w\bigl(k\bigr) = \sigma_w(K) \cap \R\quad\text{or}\quad \sigma_w(K) \cap \R = \sigma_w(K).
\end{equation}
As $K$ is totally complex the equality on the right of (\ref{extra572}) is clearly impossible, and we conclude that the equality on the left of 
(\ref{extra572}) must hold.  This verifies the equality in (\ref{extra516}).  It also shows that the restriction of complex conjugation to $\sigma_w(K)$
fixes the real subfield $\sigma_w\bigl(k\bigr)$.  The field extension $\sigma_w(K) / \sigma_w\bigl(k\bigr)$ is quadratic and therefore Galois. We conclude that $\rho$ in 
(\ref{extra523})  is indeed an automorphism.
\end{proof}

In the statement of Lemma \ref{lemhaar2} the map $\rho : \C \rightarrow \C$ is restricted in (\ref{extra523}) to the various embeddings of $K$ 
into $\C$.   We record a variant in which $K$ is fixed and the group $\Aut\bigl(K/k\bigr)$ is identified.

\begin{corollary}\label{corhaar1}  Let $K$ be an algebraic number field that satisfies the hypotheses of 
{\rm Lemma \ref{lemhaar2}}. 
 For
each archimedean place $w$ of $K$ let $\xi^{(w)}$ be a nonzero element of $O_K$
that satisfies {\rm (\ref{extra498})} from {\rm Lemma \ref{lemhaar2}}, and suppose
\begin{equation*}
k^{(w)} = \Q\bigl(\xi^{(w)}\bigr)
\end{equation*}
is a proper subfield of $K$.
Then $K/k^{(w)}$ is a Galois extension of order $2$ and
\begin{equation*}
\Aut\bigl(K/k^{(w)}\bigr) = \langle \sigma_w^{-1} \rho \sigma_w\rangle.
\end{equation*}
\end{corollary}

\begin{proof}  
It follows from Lemma \ref{lemhaar2} that $K/k^{(w)}$ is an extension of degree $2$.
And it follows from (\ref{extra523}) that
\begin{equation}\label{extra617}
\sigma_w^{-1} \rho \sigma_w : K \rightarrow K
\end{equation}
is an automorphism that fixes the subfield $k^{(w)} \subseteq K$. We conclude that
$K/k^{(w)}$ is a Galois extension and 
$\Aut\bigl(K/k^{(w)}\bigr)$ is generated by the nontrivial automorphism (\ref{extra617}).
\end{proof}

Next we prove an elementary lemma from Galois theory.

\begin{lemma}\label{lemhaar3}  Let $L$ be an algebraic number field and let
\begin{equation*}
\big\{\ell_1, \ell_2, \cdots , \ell_J\big\}
\end{equation*}
be a finite collection of subfields of $L$.  Assume that $L/\ell_j$ is a Galois extension for each $j = 1, 2, \dots , J$.  If
\begin{equation*}
F = \ell_1 \cap \ell_2 \cap \cdots \cap \ell_J
\end{equation*}
then $L/F$ is a Galois extension.  Moreover, let
\begin{equation*}
H_j = \Aut(L/\ell_j) \subseteq \Aut(L/F)\quad\text{for $j = 1, 2, \dots , J$,}
\end{equation*}  
be the subgroups attached to the extensions $L/\ell_j$.  And let
\begin{equation*}
G = \langle H_1, H_2, \cdots , H_J\rangle \subseteq \Aut(L/F)
\end{equation*}
be the group generated by the collection of subgroups $H_1, H_2, \dots , H_J$.  Then we have
\begin{equation}\label{extra218}
G = \Aut(L/F).
\end{equation}
\end{lemma}

\begin{proof}  Since the characteristic of $F$ is zero the extension $L/F$ is separable.
Suppose $\alpha$ is an  element of $L$ satisfying $L = F(\alpha)$.  Then let $P(x)$ be the monic polynomial  
\begin{equation*}
P(x) = \prod_{\gamma \in G} \bigl(x - \gamma(\alpha)\bigr).
\end{equation*}
It follows that $P(\alpha) = 0$. Since $G \subseteq \Aut(L/F)$ each root $\gamma(\alpha)$ belongs to $L$.  

Let $\vphi$ be an automorphism in $G$.  Then we have
\begin{equation}\label{extra239}
\begin{split}
\vphi\bigl(P(x)\bigr) &= \prod_{\gamma \in G} \bigl(x - \vphi\bigl(\gamma(\alpha)\bigr)\bigr)\\ 
			      &= \prod_{\gamma \in G} \bigl(x - \gamma(\alpha)\bigr) = P(x).
\end{split}
\end{equation}
From (\ref{extra239}) we conclude that each coefficient of $P(x)$ is fixed by the automorphisms in each of the groups $H_j$ for 
$j = 1, 2, \dots , J$.  Therefore each coefficient belongs to the field
\begin{equation*}\label{extra245}
\ell_1 \cap \ell_2 \cap \cdots \cap \ell_J = F.
\end{equation*}
We conclude that  $P(x)$ is in $F[x]$ and $L$ is the splitting field of $P$ over $F$.  
It follows that $L/F$ is
a Galois extension and $\#G=\deg P\geq [L:F]=\#\Aut(L/F)$. This implies that  $G=\Aut(L/F)$.
\end{proof}

\section{Proof of Theorem \ref{thmearly2}}\label{Tmain1}

If $K$ has an embedding into $\R$ then the inequality (\ref{early22}) follows from (\ref{early16}) in the statement of Theorem \ref{thmearly1}.  
Therefore we assume throughout the remainder of the proof that $K$ is totally complex.  

If $K$ is a complex quadratic field then we have $K = \Q(\mu)$ because $\mu$ is not totally real.  Since $c_K\geq 1$ it follows that (\ref{early22}) holds with
$\alpha = \mu$. 

For the remainder of the proof we assume that  the totally complex number field $K$ has degree at least
$4$, and so has at least two archimedean places.
 Let $w$ be an archimedean 
place of $K$ such that $\sigma_w(\mu)\notin \R$.  

We apply Lemma \ref{lemhaar1} with the choices $B_v=\max\{1,|\mu|_v\}^{-1}$ (for $v|\infty, v\neq w$)
to get a nonzero algebraic integer  $\xi^{(w)}$ in $O_K$ that satisfies the inequalities (\ref{extra338}) and (\ref{extra345}) in Lemma \ref{lemhaar1}.  
Consider the subfield
\begin{equation}\label{temp30}
\Q\bigl(\xi^{(w)}\bigr) \subseteq K.
\end{equation}
 Using (\ref{extra345}) we find that
\begin{equation*}
H\bigl(\xi^{(w)}\bigr) = \bigl |\xi^{(w)}\bigr |_w \le  c_K \prod_{\substack{v\mid \infty \\ v\neq w}}\max\{1,|\mu|_v\}\leq c_K H(\mu).
\end{equation*}
If there is equality in the containment (\ref{temp30}) then 
(\ref{early22}) holds with
$\alpha = \xi^{(w)}$.

Now suppose $\Q\bigl(\xi^{(w)}\bigr)$ is a proper subfield of $K$.
It follows from (\ref{extra516}) that $\sigma_w(\xi^{(w)})\in \R$.
Consider the nonzero algebraic integer $\alpha=\mu \xi^{(w)}$, and set
\begin{equation*}
k = \Q(\alpha)  \subseteq K.
\end{equation*}
We note that 
\begin{equation*}
|\alpha|_v=|\mu|_v |\xi^{(w)}|_v < \frac{|\mu|_v}{\max\{1,|\mu|_v\}}\leq 1 \quad\text{if $v | \infty$ and $v \not= w$.}
\end{equation*}
Since $\sigma_w(\xi^{(w)})\in \R$ and $\sigma_w(\mu)\notin \R$ we conclude that $\sigma_w(\alpha)=\sigma_w(\mu)\sigma_w(\xi^{(w)})\notin \R$.
Therefore Lemma \ref{lemhaar2}, applied to $\alpha$ instead of  $\xi^{(w)}$, gives $k=K$.
Finally, using that $|\alpha|_v\leq 1$ for all places $v$ with $v\neq w$, we get 
\begin{alignat*}1
H(\alpha)=|\alpha|_w=|\mu|_w|\xi^{(w)}|_w 
&\leq  |\mu|_w c_K \prod_{\substack{v\mid \infty \\ v\neq w}}\max\{1,|\mu|_v\}\\
&\leq  c_K \prod_{v\mid \infty}\max\{1,|\mu|_v\} = c_K H(\mu).
\end{alignat*}

\section{Proof of Theorem \ref{thmearly3}}\label{Tmain2}

If $K$ has a real embedding then it follows from Theorem \ref{thmearly1} that the inequality (\ref{early35}) is false.  

Now suppose that $K$ has only complex embeddings and there exists $\mu \in O_K$ such that $H(\mu) \leq \hH$ and $\mu \notin F$.  
It follows that $\mu$ is not totally real.  And Theorem \ref{thmearly2} implies that there exists $\alpha$ in $O_K$ such that 
\begin{equation}\label{extra711}
K = \Q(\alpha),\quad\text{and}\quad H(\alpha) \le H(\mu) c_K \le \hH c_K.
\end{equation}
But the inequality (\ref{extra711}) contradicts the hypothesis (\ref{early35}) in the statement of Theorem \ref{thmearly3}.  We conclude that the
hypothesis (\ref{early35}) implies that $K$ is totally complex and that every $\mu \in O_K$ with $H(\mu) \leq \hH$ lies in $F$.  
It remains to prove that (\ref{early35}) also implies that $K/F$ is a Galois extension.

If the totally complex field  $K$ has exactly one archimedean place then $K$ is an imaginary quadratic extension of $\Q$.  In this case it is easy to
prove that $F = \Q$ and $K/\Q$ is Galois. For the remainder of the proof, we assume  that $K$ has at least two archimedean places,
that satisfy 
\begin{equation*}
\{\alpha\in O_K : H(\alpha) \leq \hH\} \subseteq F,
\end{equation*}
and satisfy the inequality (\ref{early35}).
Therefore $K$ satisfies the hypotheses of 
Lemma \ref{lemhaar1} and Lemma \ref{lemhaar2}.  For each archimedean place $w$ of $K$ we select a  nonzero  $\xi^{(w)}\in O_K$ that 
satisfies the inequalities (\ref{extra338}) and (\ref{extra345}) in the statement of Lemma \ref{lemhaar1}
for the choices $B_v=1$ (for $v\mid \infty$, $v\neq w$).  It follows from Lemma \ref{lemhaar1}
that $H(\xi^{(w)})\leq c_K\leq c_K\hH$.
Therefore the hypothesis (\ref{early35}) implies that
\begin{equation}\label{extra719}
k^{(w)} = \Q\bigl(\xi^{(w)}\bigr)
\end{equation}
is a proper subfield of $K$.
Hence, it follows from Corollary \ref{corhaar1} that $K/k^{(w)}$ is a 
Galois extension of order $2$ and
\begin{equation}\label{extra726}
\Aut\bigl(K/k^{(w)}\bigr) = \langle \sigma_w^{-1} \rho \sigma_w\rangle.
\end{equation}

Next we define the subfield
\begin{equation}\label{extra748}
\widetilde{F} = \bigcap_{w | \infty} k^{(w)} \subseteq K.
\end{equation}
Then it follows from Lemma \ref{lemhaar3} that $K/\widetilde{F}$ is a Galois extension and
\begin{equation}\label{extra755}
\Aut(K/\widetilde{F}) = \langle \sigma_w^{-1} \rho \sigma_w : \text{$w | \infty$ is a place of $K$}\rangle.
\end{equation}
If $\alpha$ belongs to $\widetilde{F}$ then it follows from (\ref{extra748}) that $\alpha$ belongs to $k^{(w)}$ at each archimedean place $w$ of $K$. 
And it follows from  (\ref{extra516}) that
\begin{equation*}
\sigma_w(\alpha) \in \sigma_w\bigl(k^{(w)}\bigr) \subseteq \R
\end{equation*}
at each archimedean place $w$ of $K$.  We conclude that $\widetilde{F}$ is totally real.

Finally, let $\beta$ be an element of $K$ that is totally real.  Then at each archimedean place $w$ of $K$ we have both 
$\sigma_w(\beta) \in \sigma_w(K)$ and $\sigma_w(\beta) \in \R$.  Applying  (\ref{extra516}) we find that
\begin{equation}\label{extra769}
\sigma_w(\beta) \in \sigma_w(K) \cap \R = \sigma_w\bigl(k^{(w)}\bigr)
\end{equation}  
at each archimedean place $w$ of $K$.  We conclude from (\ref{extra769}) that
\begin{equation*}
\beta \in \bigcap_{w | \infty} k^{(w)} = \widetilde{F},
\end{equation*} 
and therefore $\widetilde{F} = F$ is the maximal totally real subfield in $K$.

\section*{Acknowledgements}

The authors are  grateful to the anonymous referees for the careful reading and  helpful  suggestions.
We thank Dr. Michael Mossinghoff for insightful conversation about this project. We are grateful to Professor Hendrik Lenstra for pointing out an error in an earlier version of this manuscript. 
The authors acknowledge support from the Max Planck Institute for Mathematics in Bonn and support from the  Institute for 
Advanced Study in  Princeton.
Shabnam Akhtari's research was partially supported by the National Science Foundation Awards DMS-2001281 and DMS-2327098.


\begin{thebibliography}{99}

\bibitem{alaca2004}
	{{\c S}.~Alaca and K.~S.~Williams}, Introductory Algebraic Number Theory, Cambridge U. Press, New York,  2004.
	

\bibitem{blanksby1978}
	 {P.~E.~Blanksby and J.~H.~Loxton}, A note on the characterization of {CM}-fields,
	 {J. Austral. Math. Soc. (Series A)}
	26:   26--30, 1978.
	
	\bibitem{bombieri2006} {E. Bombieri and W. Gubler}, {Heights in Diophantine Geometry},
	{Cambridge U. Press},  {New York},
		2006.
	
	
	\bibitem{huardspearmanwilliams1995}
	 {J.~G.~Huard and B.~K.~Spearman and K.~S.~Williams}, {Integral bases for quartic fields with  quadratic subfields},
	 {J. Number Theory},
	51:  {87--102}, 1995.
	
	\bibitem{Narkieweicz2004}
	{W.~Narkieweicz}, {Elementary and Analytic Theory of Algebraic Numbers}, 3rd ed., Springer, 2004.
	
	
	\bibitem{neukirch1999}
	 {J.~Neukirch}, {Algebraic Number Theory},
	{Springer-Verlag}, {New York},
	1999.
	
	\bibitem{northcott1949} D.~G.~Northcott. An inequality on the theory of arithmetic on algebraic varieties. Proc. Cambridge Philos. Soc., 45:  502--509, 1949.
	
	
	\bibitem{ruppert1998}
	 {W.~Ruppert},  {Small generators of number fields},
	{Manuscripta Math.},
	96(1):	{17--22},
	 {1998}.
	
	
	\bibitem{shimura1971}
	 {G.~Shimura}, {Introduction to the Arithmetic Theory of Automorphic Functions}
	{Princeton University Press}, Princeton, NJ, 
	1971.
	
	\bibitem{vaaler2013}
	{J.~D.~Vaaler and M.~Widmer}, {A note on generators of number fields}, In
	 {Diophantine Methods, Lattices and the Arithmetic Theory of Quadratic Forms}, Vol. {587} of  {Contemp. Math.}
	{201--211},
	 {Amer. Math. Soc.},  {Providence, RI} 2013.
	 
	\bibitem{Widmer2025}
	 {M.~Widmer}, {Small generators of abelian number fields},
	 {Forum Math.},
	37:  {629--636}, 2025.
	
	\bibitem{PazukiWidmer2021}
	 {F.~Pazuki and M.~Widmer}, {Bertini and Northcott},
	 {Res. Number Theory},
	7: 2021, no. 1, art. 12, 18 pp.
	
		
\end{thebibliography}
\end{document}